\documentclass[preprint, 12pt]{elsarticle}

\usepackage{amssymb}
\usepackage{amsmath}
 \usepackage{amsthm}
 \usepackage{tikz -inet}
\usetikzlibrary{positioning,shapes.misc, patterns,arrows,decorations,decorations.pathreplacing, snakes}

\newtheorem{theorem}{Theorem}
\newtheorem{assumption}{Assumption}
\newtheorem{definition}{Definition}
\newtheorem{proposition}{Proposition}
\newtheorem{fact}{Fact}

\newtheorem{remark}{Remark}
\newtheorem{corollary}{Corollary}
\newcommand{\K}{\mathcal{K}}
\newcommand{\Union}{\bigcup}
\DeclareMathOperator{\tp}{ga-tp}
\DeclareMathOperator{\id}{id}
\DeclareMathOperator{\Aut}{Aut}

\newcommand{\gaS}{\operatorname{ga-S}}
\DeclareMathOperator{\LS}{LS}
\newcommand{\T}{\mathcal{T}}
\newcommand{\C}{\mathfrak{C}}
\newcommand{\Hanf}{\operatorname{Hanf}}

\usepackage{mathtools}

\newcommand{\verteq}{\rotatebox{90}{$\,=$}}
\newcommand{\equalto}[2]{\underset{\scriptstyle\overset{\mkern4mu\verteq}{#2}}{#1}}

\journal{Annals of Pure and Applied Logic}

\begin{document}

\begin{frontmatter}



\title{Superstability and Symmetry}


\author{Monica M. VanDieren\fnref{fn1}}
\fntext[fn1]{The author was partially sponsored for this work by grant
  DMS 0801313  of the \emph{National Science Foundation.}}
\address{Robert Morris University \\  6001 University Blvd \\ Moon Township PA 15108}

\begin{abstract}

This paper continues the study of superstability in abstract elementary classes (AECs) satisfying the amalgamation property.  In particular, we consider the definition of $\mu$-superstability which is based on the local character characterization of superstability from first order logic.
Not only is $\mu$-superstability a potential dividing line in the classification theory for AECs, but it is also a tool in proving instances of Shelah's Categoricity Conjecture.

In this paper, we introduce a formulation, involving towers, of symmetry over limit models for $\mu$-superstable abstract elementary classes.  We use this formulation to gain insight into the problem of the uniqueness of limit models for categorical AECs.


\end{abstract}




\end{frontmatter}



\section{Introduction}
Finding an independence relation for abstract elementary classes (AECs)  that satisfies symmetry is a long-standing problem.  Oftentimes symmetry is just assumed as an axiom (e.g. Shelah's good frames \cite{Sh 600}).   
Up until now, when it is derived, it is usually  under additional  model-theoretic assumptions such as tameness and/or set-theoretic assumptions such as  or the existence of large cardinals.  Some examples of the contexts and assumptions are tameness and the  extra assumption of the extension property \cite{BoGr}; tameness and superstability assumptions above a sufficiently large cardinal \cite{V}; $2^\lambda<2^{\lambda^+}$, weak GCH, and categoricity in several successive cardinalities \cite{Sh 576}; and $L_{\kappa,\omega}$-theories where  $\kappa$ a strongly compact cardinal \cite{MaSh}.  
In all of these examples symmetry is developed through a forking calculus.  We provide a mechanism for deriving symmetry in abstract elementary classes without having to assume set-theoretic assumptions or tameness, and
our methods of using towers to derive symmetry differ from these earlier results.  Additionally, our point of view is localized as we consider only superstability and symmetry in $\mu$ and $\mu^+$.

 In this paper we identify a necessary and sufficient condition for symmetry of non-$\mu$-splitting over limit models (see Definition \ref{sym defn}) that involves reduced towers (Theorem \ref{symmetry theorem}).    This new condition is particularly interesting since it does not have a pre-established first-order analog.  
 
 A preliminary test for our definition of symmetry is to derive it from the assumption of categoricity:
 \begin{corollary}\label{categoricity corollary}
Suppose that $\K$ satisfies the amalgamation and joint embedding properties and $\mu$ is a cardinal $\geq\beth_{(2^{\Hanf(\K)})^+}$.  
If $\K$ is categorical in $\lambda=\mu^+$, then $\K$ has symmetry for non-$\mu$-splitting over limit models.
\end{corollary}


This examination of symmetry occurs in the context of superstable abstract elementary classes.  The sharp dividing line for superstability in the classification theory of AECs is still uncertain in general.  Grossberg and Vasey discuss the different shades of superstability and prove that under tameness they are all equivalent \cite{GV}.  This paper helps to illuminate the dividing line in classes that may fail to be tame by considering the local character characterization of $\mu$-superstability (see Assumption \ref{ss assm}).  
In particular, we 
extend the work of \cite{ShVi}, \cite{Va1}, \cite{Va2}, and \cite{GVV} 
 to derive the uniqueness of limit models of cardinality $\mu$  in $\mu$-superstable AECs which satisfy $\mu$-symmetry over limit models.

The series of papers \cite{ShVi}, \cite{Va1}, \cite{Va2}, and \cite{GVV} aims to verify the still-open conjecture that the uniqueness of limit models of cardinality $\mu$ is equivalent to superstability.  
Grossberg and Boney take a different approach to the uniqueness of limit models  under the additional assumptions of tameness and a symmetry condition \cite{BoGr}.
We refer the reader to \cite{GVV} and \cite{Va1} for a literature review on the  uniqueness of limit models in abstract elementary classes.  Briefly, the uniqueness of limit models has been used as a step to prove instances of Shelah's Categoricity Conjecture (e.g. \cite{Sh 394}, \cite{GV1}, \cite{GV2}), 
as a mechanism for finding (Galois-)saturated models in singular cardinals \cite{Sh 394},
as a tool to derive the amalgamation property \cite{KoSh}, and as a lens to examine strictly stable AECs \cite{BV}.

The organization of the proofs in the series  \cite{ShVi}, \cite{Va1}, \cite{Va2}, and \cite{GVV}  is to transfer some model-theoretic structure from assumptions in a higher cardinality down to $\mu$. 
Here, in Theorem \ref{main theorem}, we extend this line of reasoning by weakening the structural assumptions on $\K_{\mu^+}$.

 \begin{theorem}\label{main theorem}
 Let $\K$ be a $\mu$-stable abstract elementary class satisfying the amalgamation and joint embedding properties.  Suppose $\K$  satisfies the locality and continuity properties of $\mu$-splitting (see Assumption \ref{ss assm}) and also satisfies the property that any union of saturated models of cardinality $\mu^+$ is saturated.  For $N\in \K_\mu$ and  limit ordinals, $\theta_1, \theta_2<\mu^+$, if $M_l$ is $(\mu,\theta_l)$-limit model for $l\in\{1,2\}$, then $M_1$ and $M_2$ are isomorphic over $N$.
 
 \end{theorem}

More interestingly, the proof of Theorem \ref{main theorem} leads to the isolation of the question of whether or not every reduced tower is continuous.  Here we see that this is equivalent to asking whether or not the abstract elementary class has a non-$\mu$-splitting relation that satisfies symmetry over limit models.  This provides some insight into the strength of the assumptions that might be required to prove the uniqueness of limit models.

Later research has established that our formulation of symmetry is equivalent to other existing forms and that symmetry for non-$\mu$-splitting over limit models follows from $\mu$-superstability in tame AECs \cite{VV}.
 Our formulation of symmetry has also been used to  derive the following results:
  $\mu$-superstability implies the uniqueness of limit models in tame AECs \cite[Corollay 1.4]{VV}; the union of an increasing chain of $\mu^+$-saturated models is $\mu^+$-saturated in $\mu$- and 
  $\mu^+$-superstable classes that satisfy symmetry for non-$\mu^+$-splitting over limit models \cite[Theorem 1]{Va3}; and some level of uniqueness of limit models can be recovered in strictly $\mu$-stable AECs which satisfy a weakening of Definition \ref{sym defn} \cite[Theorem 1]{BV}. Moreover, Vasey has used our results to lower the bound of Shelah's Downward Categoricity Transfer Theorem \cite[Corollary 7.12]{Vas-categoricity}.

The author would like to thank John Baldwin and Chris Laskowski for conversations about symmetry for non-splitting after her presentation of an early version of Theorem \ref{symmetry theorem} at the BIRS Neo-stability Theory Workshop in 2012.
The author is also grateful for constructive and valuable feedback from Rami Grossberg, Sebastien Vasey, Will Boney, and the referee on earlier drafts of this paper.

\section{Background}
Much of the notation and definitions used here is from \cite{GVV}.  We remind the reader of a few of these.  We fix $\K$ an abstract elementary class with relation $\prec_{\K}$.  For $\mu$ a cardinal $\geq$ the L\"{o}wenheim-Skolem number of the class ($\LS(\K)$), we write $\K_\mu$ for the class of models of cardinality $\mu$.  Much of our focus will be on limit models. For $\alpha$ a limit ordinal $<\mu^+$, a  \emph{$(\mu,\alpha)$-limit model} is a model $M$ which is the union of an increasing sequence $\langle M_i\in\K_\mu\mid i<\alpha\rangle$ for which $M_{i+1}$ is universal over $M_i$ (in other words for every $N\in\K_\mu$ with $M_i\prec_{\K}N$ there exists a $\K$-embedding $f:N\rightarrow M_{i+1}$ with $f\restriction M_i=\id_{M_i}$).  
In this case we also say that $M$ is a $(\mu,\alpha)$-limit model over $M_0$.
When $\alpha$ is not important or is otherwise clear, we abbreviate $(\mu,\alpha)-$limit model by limit model.  We write $\K^*_\mu$ for $$\{M\in\K_\mu\mid\exists\alpha<\mu^+\text{a limit ordinal so that }M\text{ is a }(\mu,\alpha)-\text{limit model}\}.$$

We improve the results of \cite{Va2} and \cite{GVV}  by weakening the assumptions used to show that reduced towers are continuous (see Theorem \ref{transfer theorem}).  
Let us establish those assumptions now.

For the remainder of this paper, we assume that $\K$ is an abstract elementary class satisfying the joint embedding and amalgamation properties and $\K$ has no maximal models of cardinality $\mu^+$.  Therefore we can fix $\C$ a homogeneous (monster-like) model of cardinality $\geq\mu^+$.
We will assume that the class $\K$ satisfies $\mu$-superstability assumptions laid out in \cite{GVV}.  Namely:
\begin{assumption}\label{ss assm}
\mbox{}
\begin{enumerate}
\item $\K$ is (Galois-)stable in $\mu$.
\item\label{split assm} $\mu$-splitting in $\K$ satisfies the following
  locality (sometimes called continuity) and ``no long splitting chains''
  properties.
%
%
For all infinite $\alpha$, for every sequence $\langle M_i\mid i<\alpha\rangle$ of
  limit models of cardinality $\mu$ with $M_{i+1}$ universal over $M_i$ and for every $p\in\gaS(M_\alpha)$, where
  $M_\alpha=\bigcup_{i<\alpha}M_i$, we have that
\begin{enumerate}
\item\label{locality} If for every $i<\alpha$, the type $p\restriction
  M_i$ does not $\mu$-split over $M_0$, then $p$ does not $\mu$-split over
  $M_0$.
\item\label{no long splitting chain} There exists $i<\alpha$ such that $p$
does not $\mu$-split over $M_i$.
\end{enumerate}
\end{enumerate}

\end{assumption}

\begin{remark}
Vasey has pointed out that by the weak-transitivity property of non-splitting (i.e. Proposition 3.7 of \cite{V2}), Assumption \ref{ss assm}.\ref{split assm}.\ref{locality} actually follows from Assumption \ref{ss assm}.\ref{split assm}.\ref{no long splitting chain}. 
\end{remark}

Note that the arguments presented here also work in classes that do not necessarily satisfy the amalgamation property such as the context covered in \cite{ShVi}, but for ease of readability, we work in the more popular setting in which the amalgamation property holds.    For details of working without the amalgamation property see \cite{V-Volume}.

Recall that with the amalgamation property along with $\mu$-stability, we can find a saturated model $M\in\K_{\mu^+}$, and  we can find  for every $N\in\K_{\mu}$, a model  $N'\in\K_{\mu}$ which is universal over $N$.  Thus it makes sense for us to work with limit models.  Furthermore, $\mu$-stability implies the existence of non-$\mu$-splitting types:
\begin{fact}[Theorem I.4.10 of \cite{Va1}]\label{splitting extension lemma}
Let $M,N,M^*$ be models in $\K_\mu$.  Suppose that $M$ is universal
over $N$ and that $M^*$ is an extension of $M$.  If a type $p=\tp(a/M)$
does not $\mu$-split over $N$ then there exists an automorphism $g$ of
$\C$ fixing $M$ such that $\tp(g(a)/M^*)$ does not $\mu$-split over
$N$ and $\tp(g(a)/M)=p$.
\end{fact}

A tower -- a construct which combines chains of limit models with non-splitting types -- is the main mechanism that is used to prove the uniqueness of limit models. 
A \emph{tower} is a sequence of length $\alpha$ of limit models, denoted by $\bar M=\langle M_i\in\K^*_\mu\mid i<\alpha\rangle$, along with a sequence of designated elements $\bar a=\langle a_{i}\in M_{i+1}\backslash M_i\mid i+1<\alpha\rangle$ and a sequence of designated submodels $\bar N=\langle N_{i}\mid i+1<\alpha\rangle$ for which
 $M_i\prec_{\K}M_{i+1}$, $\tp(a_i/M_i)$ does not $\mu$-split over $N_i$, and $M_i$ is universal over $N_i$ (see Definition I.5.1 of \cite{Va1}).  
 We refer to $\alpha$ as the length of the tower.
 The class of all towers indexed by an ordinal $\alpha$ and made up of limit models of cardinality $\mu$ is denoted by $\K^*_{\mu,\alpha}$.
 
 Notice that the sequence $\bar M$ in the definition of tower is not required to be continuous (by \emph{continuous} we mean that for every limit ordinal $i<\alpha$ $M_i=\Union_{j<i}M_j$).  In fact, many times we will not have continuous towers. For instance, 
 discontinuous towers arise in Fact \ref{thm:extension for towers} when we find extensions of towers.  We should make this definition of extension more explicit since it will be used later in this paper.  This definition of extension appears in \cite{ShVi} and \cite{Va1} where the notation $<^c$ was used to distinguish it from other versions of extension.

\begin{definition}
For towers $(\bar M,\bar a,\bar N)$ and $(\bar M',\bar a',\bar N')$ in $\K^*_{\mu,\alpha}$, we say $$(\bar M,\bar a,\bar N)\leq (\bar M',\bar a',\bar N')$$ if $\bar a=\bar a'$, $\bar N=\bar N'$, $M_i\preceq_{\K}M'_i$, and whenever $M'_i$ is a proper extension of $M_i$, then $M'_i$ is universal over $M_i$.  If for each $i<\alpha$,  $M'_i $ is universal over $M_i$ we will write $(\bar M,\bar a,\bar N)< (\bar M',\bar a',\bar N')$, and we will refer to this relationship between towers as proper extension.
\end{definition}

Notice that even if $(\bar M,\bar a,\bar N)$ is continuous, it may be that the only extensions of $(\bar M,\bar a,\bar N)$ are discontinuous towers; though, we will rule this out in some contexts.  In particular, we will see that when Assumption \ref{ss assm} and $\mu$-symmetry over limit models hold, then all towers -- whether discontinuous or continuous --  have continuous proper extensions (Fact \ref{density of reduced} combined with Theorem \ref{symmetry theorem}).

Before proceeding, we need to establish some notation about towers that will simplify the presentation. We will often abbreviate the tower $(\bar M,\bar a,\bar N)$ by $\T$.  
If $(\bar M,\bar a,\bar N)$ is a tower and $f$ is an automorphism of $\C$, we will write $f(\T)$ for the image of the tower under $f$.  We often construct towers in stages and  will have to refer to towers of varying length and to initial segments of towers.  Therefore for a tower $\T=(\bar M,\bar a,\bar N)\in\K^*_{\mu,\alpha}$ and for $\beta$ an ordinal $<\alpha$, we write $\T\restriction\beta$ or $(\bar M,\bar a,\bar N)\restriction\beta$ for the tower made up the sequences $\langle M_i\mid i<\beta\rangle$, $\langle a_i\mid i+1<\beta\rangle$ and $\langle N_i\mid i+1<\beta\rangle$.

We also need to  recall a few facts about directed system of partial extensions of towers from \cite{Va1}.  We restate them here in the form they will be used.  The first fact will get us through the successor step of directed system inductive constructions. 
\begin{fact}\label{direct limit prop}
Suppose $\T$ is a tower in $\K^*_{\mu,\alpha}$ and $\T'$ is a tower of length $\beta<\alpha$ with $\T\restriction \beta<\T'$, if $f\in\Aut_{M_\beta}(\C)$ and $M''_\beta$ is a limit model universal over $M_{\beta}$ such that $\tp(a_\beta/M''_\beta)$ does not $\mu$-split over $N_\beta$ and $f(\Union_{i<\beta}M'_i)\prec_{\K}M''_\beta$, then the tower $\T''\in\K^*_{\mu,\beta+1}$ defined by $f(\T')$ concatenated with the model $M''_\beta$, element $a_\beta$ and submodel $N_\beta$ is an extension of $\T\restriction (\beta+1)$.
\end{fact}
Fact \ref{direct limit prop} follows from  the definition of tower extensions.  

The next fact describes how to pass through the limit stages. For completeness, we provide its proof.

\begin{fact}\label{limit stage prop}
Fix $\T\in\K^*_{\mu,\alpha}$ for $\alpha$ a limit ordinal.
Suppose $\langle \T^i\in\K^*_{\mu,i}\mid i<\alpha\rangle$  and $\langle f_{i,j}\mid i\leq j<\alpha\rangle$ form a directed system of towers that satisfy the following conditions:
\begin{enumerate}
\item each $\T^i$ extends $\T\restriction i$
\item\label{identity condition} $f_{i,j}\restriction M_i=id_{M_i}$
\item\label{universal over condition} $f_{i,i+1}(\Union_{l<i}M^i_l)\prec_{\K}M^{i+1}_i$.
\end{enumerate}
Then there exists a direct limit $\T^\alpha$ and mappings $\langle f_{i,\alpha}\mid i<\alpha\rangle$ to this system so that $\T^\alpha\in\K^*_{\mu,\alpha}$; $\T<\T^\alpha$; and $f_{i,\alpha}\restriction M_i=id_{M_i}$.  
\end{fact}

\begin{proof}
To take a direct limit of towers, we first consider the continuous directed system of models $\langle M^*_i\mid i<\alpha\rangle$ and $\langle f_{i,j}\mid i\leq j<\alpha\rangle$, where we define $M^*_i:=M^{i+1}_i$ when $i$ is a successor and $M^*_i:=\Union_{k<i}M^{k+1}_k$ when $i$ is a limit ordinal.  The reasons for this nuanced definition are that the sequences $\bar M^i$ may not be continuous and that $\T^i\in\K^*_{\mu,i}$ is a tower whose maximally indexed model has index $<i$.
By the amalgamation property and condition \ref{identity condition}, we can choose the direct limit $\Union_{i<\alpha}f_{i,\alpha}(M^*_i)$ to be an extension of $\Union_{i<\alpha}M_i$ with $f_{i,\alpha}\restriction M_i=\id_{M_i}$.  

To prepare for the construction of a tower from this direct limit, notice that $f_{i+1,\alpha}(a_i)=a_i$ and $f_{i,\alpha}(N_i)=N_i$.  Furthermore $\tp(a_i/f_{i+1,\alpha}(M^{i+1}_i))$ does not $\mu$-split over $N_i$ by invariance and our assumption that $\T^{i+1}$ is a tower.
Notice that the sequence $\langle f_{i,\alpha}(M^{i+1}_i)\mid i<\alpha\rangle$ is increasing.  

We can now use this direct limit to construct a tower, $\T^\alpha$.  For $i<\alpha$, set $M^\alpha_i:=f_{i,\alpha}(M^{i+1}_i)$.   This definition is called for because for limit ordinals $i$, $\Union_{j<i}M^{j+1}_j$ may not be universal over $M_i$.  However, this causes no issues because $\Union_{j<i}f_{j,\alpha}(M^{j+1}_j)\prec_{\K}f_{i+1,\alpha}(M^{i+1}_i)$ by condition \ref{universal over condition} and the requirement of directed systems that for $j+1<k<i$, $f_{k,i}(f_{j+1,k}(M^{j+1}_j))=f_{j+1,i}(M^{j+1}_j)$.
 Since all the models $M^\alpha_i$ are limit and the non-splitting condition is met, $\T^\alpha:=(\bar M^\alpha,\bar a,\bar N)$ is a (possibly discontinuous) tower.  See Figure \ref{fig:limit construction}.

\begin{figure}[h]
\begin{tikzpicture}[scale=2.35,inner sep=.5mm]
\tikzstyle{rrect}=[rounded corners=5mm]
\draw (0,1) [rrect] rectangle (3.5,0);
\draw (0,1)[rrect] rectangle (1.1,-1.5);
\draw (0,1)[rrect] rectangle (1.1,0);
\draw (0,1)[rrect] rectangle (2.1,-1.5);
\draw (0,1)[rrect] rectangle (2.1,0);
\draw (0,1)[rrect] rectangle (2.55,0);
\draw (1.1,0) [dotted] to (1.25,-.75);
\draw (2.1,0)[dotted] to (2.25,-.75);
\draw (2.55,.2) [dotted]to (2.7,-.75);
\draw (.85,.75) node {$M_{0}$};
\draw (1.75,.75) node {$\dots  M_{i}$};
\draw (2.35,.75) node {$M_{i+1}$};
\draw (3,.65) node {$\dots \displaystyle{\Union_{k<\alpha}M_k}$};
\draw (4.45,-.5) node {$\T^\alpha$};
\draw (2.35,-.5) node {$M^\alpha_i$};
\draw (1.35, -.5) node {$M^\alpha_0$};
\draw (3, -.75) node {$\equalto{M^\alpha_{i+1}}{f_{i+2,\alpha}(M^{i+2}_{i+1})}$};
\draw (-.25,.15) node {$\T$};
\draw [<-, shorten >=3pt] (2.3,-.75) to [bend left=65] node[pos=0.5,right] {$f_{i+1,\alpha}$}(2,-1.5);
\draw (.85,-1.35) node {$M^{i+1}_{0}$};
\draw (1.75,-1.35) node {$\dots M^{i+1}_{i}$};
\draw (-.25,-1.25) node {$\T^{i+1}$};
\node at (2.2,.25)[circle, fill, draw, label=315:$ a_{i}$] {};
\node at (1.2,.25)[circle, fill, draw, label=315:$a_{0}$] {};
\begin{scope}
  \clip (0,1) [rrect] rectangle (5,-1);
\draw (-.2,1) [rrect, xslant=-0.4, dotted] rectangle (4, -.75); 
\end{scope}

\end{tikzpicture}
\caption{The dotted figure depicts the relationship of the direct limit, $\T^\alpha$, to a directed system of towers at stage $i$.} \label{fig:limit construction}
\end{figure}
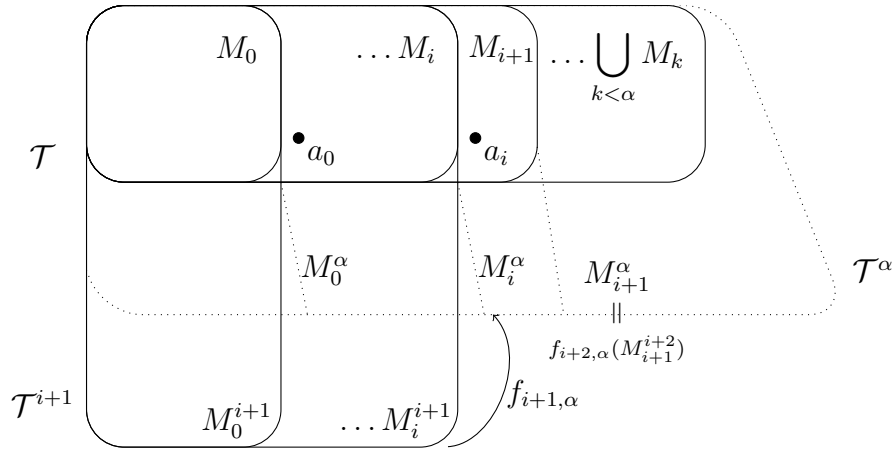

To verify that $\T^\alpha >\T$ we need only check that $f_{i+1,\alpha}(M^{i+1}_i)$ is universal over $M_i$.   This follows from the assumption that $M^{i+1}_i$ is universal over $M_i$ and from 
our selection of the direct limit so that
$f_{i+1,\alpha}(M_i)=M_i$. 
\end{proof}

There are a couple of things to note concerning Fact \ref{limit stage prop}.  First, $\T^\alpha$ need not be continuous, since the towers $\T^i$ from which $\T^\alpha$ is formed may not be continuous.
Also, we do not require that at the top of the tower, the model $\Union_{i<\alpha}M^\alpha_i$,  to be universal over $\Union_{i<\alpha}M_i$.  This is because the definition of extension for towers of length $\alpha$ only requires the universality to hold for models indexed by $i<\alpha$ along the tower.

The first application of Fact \ref{direct limit prop} and Fact \ref{limit stage prop} is the extension property for towers which appears with proof in \cite{GVV}.  Assumption \ref{ss assm} is all that is called upon in the proof given in \cite{GVV}.  We provide the detailed proof here for completeness.

\begin{fact}[Lemma 5.3 of \cite{GVV}]\label{thm:extension for towers}
Given $\T\in\K^*_{\mu,\alpha}$ there exists $\T'\in\K^*_{\mu,\alpha}$ so that $\T<\T'$.

\end{fact}
\begin{proof}
The construction of $\T'$ extending $\T$ is by way of a directed system of partial extensions defined inductively.  For every $i<\alpha$, we define $\T^i\in\K^*_{\mu,i}$ an extension of $\T\restriction i$ and a system of mappings $\langle f_{i,j}\mid i\leq j<\alpha\rangle$ so that $f_{i,j}\restriction M_i=\id_{M_i}$, $f_{i,i}=\id$, and $M^{i+1}_i$ is universal over $f_{i,i+1}(\Union_{j<i}M^i_j)$.

For the base cases, $\T^0$ is the empty tower and $\T^1$ is a trivial tower, $(M^0_0, \langle\rangle,\langle\rangle)$, where $M^0_0$ is just a limit model universal over $M_0$ and $f_{0,0}=\id$.

Fact \ref{limit stage prop} describes how to carry out the limit stage of the construction.

We consider two cases when building the towers and mappings at the successor stage.
First suppose that $i=j+1$ where $j$ is a successor ($j=k+1$) and that $\T^j$ and $\langle f_{k,j}\mid k\leq j\rangle $ have been defined.  
Fix $M^*_j$ a limit model universal over both $M^j_k$ and $M_j$.
By Fact  \ref{splitting extension lemma}, there exists $f_{j,j+1}\in\Aut_{M_j}(\C)$ so that $\tp(a_j/f_{j,j+1}(M^*_j))$ does not $\mu$-split over $N_j$.  Set $M^i_j:=f_{j,j+1}(M^*_j)$ and for $l<j$, set $M^i_l:=f_{j,j+1}(M^j_l)$.  Fact \ref{direct limit prop} tells us that the tower induced by the sequence $\bar M^{j+1}$ is an extension of $\T\restriction (j+1)$ as required.  We finish the construction of this stage by letting $f_{j+1,l}:=f_{j,j+1}\circ f_{l,j}$ for all $l<j$.

The construction when $i$ is 
the successor of a limit ordinal $j$ is slightly different (although it is actually simpler than the other two cases).  We provide the details.  
Suppose $j$ is a limit ordinal and $\T^j$ and $\langle f_{l,k}\mid l\leq k< j\rangle$ have been defined satisfying the conditions of the construction.  First, set $M^{j+1}_j$ to be a limit model of cardinality $\mu$ which  is universal over both $\Union_{l<j}M^j_l$ and  $M_j$.  Also for $l<j$, let $M^{j+1}_l:=M^j_l$.  Then setting $f_{l,j+1}:=f_{l,j}$ for $l<j$ and $f_{j,j+1}:=\id$ finishes the definition of the $j+1$ stage of the construction.  Notice that the definition of extension property for a tower in $\K^*_{\mu,j+1}$ requires only that the model indexed by $j$ is universal over $M_j$ and that there are no requirements about non-splitting over $M^{j+1}_j$ (non-splitting with regard to $a_j$ is handled in the $j+2$ stage  via the construction in the previous paragraph).

This completes the construction, and an application of Fact \ref{limit stage prop} produces the required $\T'$.

\end{proof}

Note that the extension found in Fact \ref{thm:extension for towers} is not continuous.

The standard strategy of proving the uniqueness of limit models involves continuous towers. 
To produce continuous towers, one may consider reduced towers and verify that, under certain hypotheses, they are both continuous  (e.g. Theorem 5.8 from \cite{GVV})  and dense in $(\K^*_{\mu,\alpha},<)$ (e.g. Theorem III.11.5 from \cite{Va1}).

\begin{definition}\label{reduced defn}\index{reduced towers}
A tower $(\bar M,\bar a,\bar N)\in\K^*_{\mu,\alpha}$ is said to 
be \emph{reduced} provided that for every $(\bar M',\bar a,\bar
N)\in\K^*_{\mu,\alpha}$ with
$(\bar M,\bar a,\bar N)\leq(\bar M',\bar a,\bar
N)$ we have that for every
$i<\alpha$,
$$(*)_i\quad M'_i\cap\Union_{j<\alpha}M_j = M_i.$$
\end{definition}

There are a few facts about reduced towers that are known to hold under Assumption \ref{ss assm}.  
The first is the statement that a union of an increasing chain of reduced towers is reduced.  

\begin{fact}\label{union of reduced is reduced}
If $\langle \T^\gamma\mid \gamma<\beta\rangle$ is a $<$-increasing chain of reduced, continuous towers, then $\Union_{\gamma<\beta}\T^\gamma$ is also reduced and continuous.
\end{fact}
At the request of the referee, we provide an outline of the proof.
First we must establish that the union of towers $(\T:=\Union_{\gamma<\beta}\T^\gamma)$ is a tower. 
To get that for $j<\alpha$, the models $M_j:=\Union_{\gamma<\beta}M^\gamma_j$ in $\T$ are limit models, we use the
definition of extension for towers which establishes that  the sequence $\langle M^\gamma_j\mid \gamma<\beta\rangle$ 
witnesses that $M_j$ is a limit model.
Assumption \ref{ss assm} gives the required  non-$\mu$-splitting of $\tp(a_j/M_j)$ over $N_j$.  Continuity of $\T$ follows by our assumption that all the towers $\T^\gamma$ were continuous.
The proof that $\T$ is also reduced is  by contradiction.  
Basically, the witness that $\T$ is not reduced would also witness that one of the $\T^\gamma$ is not reduced.

The following is  a special case of  Theorem 3.1.13 of \cite{ShVi}.  It appears in this form with complete proof as Theorem 5.6 of \cite{GVV}.   The proof of Theorem 5.6 only requires  
Fact \ref{thm:extension for towers} and
that the union of a $<$-increasing chain of towers is a tower.  The proof is by contradiction.  Using Fact \ref{thm:extension for towers}, build a long ($\mu^+$) chain of towers in $\K^*_{\mu,\alpha}$ 
with $\T^{i+1}$ witnessing that $\T^i$ is not
 reduced. The proof involves a straightforward counting argument involving the indices $\gamma<\alpha$ where the tower $\T^i$ fails to be reduced.

\begin{fact}[Density of reduced towers]\label{density of reduced}
For every  $\T\in\K^*_{\mu,\alpha}$ there exists $\T'\in\K^*_{\mu,\alpha}$
which is a 
$<$-extension of $\T$ and is reduced.  
\end{fact}
The reduced extension found in Fact \ref{density of reduced} is not immediately known to be continuous.

Finally we will need the following monotonicity property of reduced towers.  Under slightly different hypotheses this appears as Lemma III.11.5 in \cite{Va1}.
\begin{proposition}\label{monotonicity}
Suppose that $(\bar M,\bar a,\bar N)\in\K^*_{\mu,\alpha}$ is
reduced. Then $(\bar M,\bar
a,\bar N)\restriction \beta$ is reduced for every $\beta<\alpha$.
\end{proposition}

\begin{proof}
Suppose the claim fails for the reduced tower $\T:=(\bar M,\bar a,\bar N)\in\K^*_{\mu,\alpha}$.  
Let $\beta<\alpha$ be such that $(\bar M,\bar a,\bar N)\restriction\beta$ is not reduced.
Fix $i<j<\beta$,  $b\in M_j\backslash M_i$, and there exists $\T'\in\K^*_{\mu,\beta}$ so that $b\in M'_i$.

We can now build a directed system of partial extensions of $\T$ as in Fact \ref{thm:extension for towers}.  
Only here we start with $\T^\gamma:=\T'\restriction \gamma$ for all $\gamma<\beta$ and
$\T^\beta:=\T'$ with $f_{\gamma,\zeta}=\id_{M'_\gamma}$ when $\gamma\leq \zeta\leq\beta$.  The rest of the construction for $\zeta\in(\beta,\alpha]$ is the same as Fact \ref{thm:extension for towers}.
Because $b\in M_j$, $b$ is fixed throughout the construction. To see this, notice that by the requirements of the construction $f_{j,\zeta}$ fixes $M_j$ so $f_{j,\zeta}(b)=b$.  Then  for $\zeta\in(\beta,\alpha]$, we have
$b\in f_{j,\zeta}(M'_i) = f_{j,\zeta}(f_{i+1,j}(M'_i))=f_{i+1,\zeta}(M'_i)=f_{i+1,\zeta}(M^{i+1}_i)=:M^\zeta_i$.  Thus, $\T^\alpha$, $b$, and $i$ witness that $\T$ is not reduced.

\end{proof}


\section{Transferring structure from $\K_{\mu^+}$ to $\K_{\mu}$}
In this section we provide an argument for Theorem \ref{main theorem}.
First notice that by the proof of Theorem 1.9 in \cite{GVV}, in order to prove Theorem \ref{main theorem} stated in the introduction, it is enough to prove that reduced towers are continuous (Theorem \ref{transfer theorem}).  Therefore we will concentrate on Theorem \ref{transfer theorem} for the remainder of this section.  The proof of Theorem \ref{transfer theorem} is similar to that of Theorem 2 of \cite{Va2}.

\begin{theorem}\label{transfer theorem}

 Let $\K$ be a $\mu$-stable abstract elementary class that  satisfies the locality and continuity properties of $\mu$-splitting.  If, in addition, $\K$ satisfies the  property that the union of any chain of saturated models of cardinality $\mu^+$ is saturated, then any reduced tower made up of models of cardinality $\mu$ is continuous.
\end{theorem}

\begin{proof}
Suppose the theorem fails.  Let $(\bar M,\bar a,\bar N)\in\K^*_{\mu,\alpha}$ be a counter-example of minimal length, $\alpha$.  Notice that by Proposition \ref{monotonicity}, we can conclude that $\alpha=\delta+1$ for some limit ordinal $\delta$ and that the failure of continuity must occur at $\delta$.
Let $b\in M_\delta\backslash \Union_{i<\delta}M_i$ witness the discontinuity of the tower.  

By the minimality of $\alpha$ and the density of reduced towers (Fact \ref{density of reduced}) we can construct a $<$-increasing and continuous chain of reduced, continuous towers $\langle (\bar M,\bar a,\bar N)^i\in\K^*_{\mu,\delta}\mid i<\mu^+\rangle$ with $(\bar M,\bar a,\bar N)^0:=(\bar M,\bar a,\bar N)\restriction \delta$. Let $\displaystyle{\check M:=\Union_{i<\mu^+,\;\beta<\delta}M^i_\beta}$.  Refer to Figure \ref{fig:Mdeltas}.

\begin{figure}[h]
\begin{tikzpicture}[rounded corners=5mm,scale =2.9,inner sep=.5mm]
\draw (0,1.5) rectangle (.75,.5);
\draw (0,1.5) rectangle (1.75,1);
\draw (.25,.75) node {$N_0$};
\draw (1.25,1.25) node {$N_i$};
\draw (0,0) rectangle (4,1.5);
\draw (0,1.5) rectangle (3.5,-2);
\draw (.85,.25) node {$M_0$};
\draw(1.25,.25) node {$M_1$};
\draw (1.75,.25) node {$\dots M_i$};
\draw (2.35,.25) node {$M_{i+1}$};
\draw (3.15,.2) node {$\dots\displaystyle{\Union_{k<\delta}M_k}$};
\draw (3.85, .25) node {$M_\delta$};
\draw (-.5,.25) node {$(\bar M,\bar a,\bar N)$};
\draw (0,1.5) rectangle (3.5, -.4);
\draw (.85,-.15) node {$M^{1}_0$};
\draw (1.75,-.15) node {$\dots M^{1}_i$};
\draw (2.35,-.15) node {$M^{1}_{i+1}$};
\draw(1.25,-.15) node {$M^{1}_1$};
\draw (3.15,-.2) node {$\dots\displaystyle{\Union_{l<\delta}M^{1}_l}$};
\draw (-.5,-.15) node {$(\bar M,\bar a,\bar N)^1$};
\draw (.85,-.6) node {$\vdots$};
\draw (1.75,-.6) node {$\vdots$};
\draw (2.35,-.6) node {$\vdots$};
\draw (3.2,-.6) node {$\vdots$};
\draw (0,1.5) rectangle (3.5, -1);
\draw (.85,-.85) node {$M^{j}_0$};
\draw (1.75,-.85) node {$\tiny{\dots} M^{j}_i$};
\draw (2.35,-.85) node {$M^{j}_{i+1}$};
\draw (3,-.85) node {$\dots\Union_{l<\delta}M^{j}_l$};
\draw (-.5,-.85) node {$(\bar M,\bar a,\bar N)^j$};
\draw (0,1.5) rectangle (3.5, -1.35);
\draw (0,1.5) rectangle (1,-2);
\draw(0,1.5) rectangle (1.5, -2);
\draw (0,1.5) rectangle (2.5, -2);
\draw (0,1.5) rectangle (2,-2);
\draw (0,1.5) rectangle (3.5, -2);
\draw (.8,-1.15) node {$M^{j+1}_0$};
\draw (1.8,-1.15) node {$ M^{j+1}_i$};
\draw (2.3,-1.15) node {$M^{j+1}_{i+1}$};
\draw (3,-1.2) node {$\dots\Union_{l<\delta}M^{j+1}_l$};
\draw (-.5,-1.15) node {$(\bar M,\bar a,\bar N)^{j+1}$};
\draw (.85,-1.6) node {$\vdots$};
\draw (1.75,-1.6) node {$\vdots$};
\draw (2.35,-1.6) node {$\vdots$};
\draw (3.2,-1.6) node {$\vdots$};
\node at (3.75,.75)[circle, fill, draw, label=90:$b$] {};
\node at (2.25,.75)[circle, fill, draw, label=290:$a_i$] {};
\node at (1.1,.75)[circle, fill, draw, label=290:$a_1$] {};
\draw (3.75,-1.75) node {$\check{M}$};
\end{tikzpicture}
\caption{$(\bar M,\bar a,\bar N)$ and the  towers $(\bar M,\bar a,\bar N)^j$ extending $(\bar M,\bar a,\bar N)\restriction\delta$.} \label{fig:Mdeltas}
\end{figure}
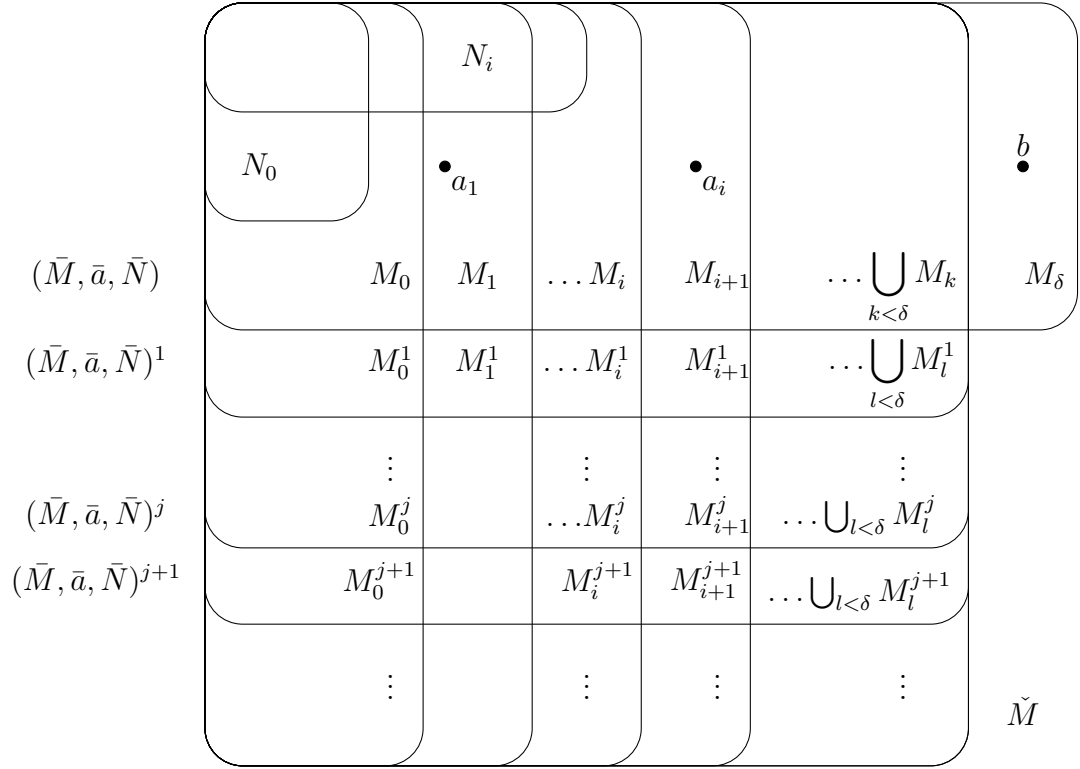

Notice that the model 
$\check M$ is a (Galois-)saturated model of cardinality $\mu^+$.  To see this, we observe that for each $\beta<\delta$, the model $\Union_{i<\mu^+}M^i_\beta$ is saturated because $M^{i+1}_\beta$ is universal over $M^i_\beta$.  Then by our assumption that the union of $\mu^+$-saturated models is saturated, $\check M$ is saturated.

Fix $\check b\in \check M$ so that  $\check b\models
\tp(b/\Union_{\beta<\delta}M_\beta)$.  Fix the minimal $i<\mu^+$ and $l+1<\delta$ so that $\check b\in M^i_{l+1}$.
By the equality of the types of $b$ and $\check b$ over $\Union_{\beta<\delta}M_\beta$, we can 
choose a $\K$-mapping $f$ so that $f(\check b)=b$ and $f\restriction\Union_{\beta<\delta}M_\beta$ is the identity.
Now consider the tower $(\bar M',\bar a,\bar N)\in\K^*_{\mu,\alpha}$ defined by setting $M'_\beta:=f(M^i_\beta)$ for $\beta<\delta$ and choosing $M'_\delta$ to be a limit model which extends $\Union_{\beta<\delta} M'_\beta$ and is universal over $M_\delta$.  See Figure \ref{fig:reduced tower construction}.  Notice that $(\bar M',\bar a,\bar N)$ and $b$ witness that $(\bar M,\bar a,\bar N)$ is not reduced.

\begin{figure}[h]
\begin{tikzpicture}[scale=2.35,inner sep=.5mm]
\tikzstyle{rrect}=[rounded corners=5mm]
\draw (0,0) [rrect] rectangle (4.25,1);
\draw (0,1) [rrect] rectangle (3.5,-1);
\draw (0,1)[rrect] rectangle (1.1,-1);
\draw (0,1)[rrect] rectangle (1.1,0);
\draw (0,1)[rrect] rectangle (2.1,-1);
\draw (0,1)[rrect] rectangle (2.1,0);
\draw (0,1)[rrect] rectangle (2.55,-1);
\draw (0,1)[rrect] rectangle (2.55,0);
\draw (1.1,-.45) [dotted] to (1.25,-.75);
\draw (2.1,-.45)[dotted] to (2.25,-.75);
\draw (2.55,-.45) [dotted]to (2.7,-.75);
\draw (.85,.75) node {$M_{0}$};
\draw (1.75,.75) node {$\dots  M_{l}$};
\draw (2.35,.75) node {$M_{l+1}$};
\draw (3,.75) node {$\dots \Union_{\beta<\delta}M_\beta$};
\draw (4, .75) node {$M_\delta$};
\draw (4,-.5) node {$M'_{l}$};
\draw (4.75,.15) node {$(\bar M',\bar a,\bar N)$};
\node at (3.7,.25)[circle, fill, draw, label=90:$b$] {};
\node at (2.45,-.9)[circle, fill, draw, label=90:$\check{b}$] {};
\draw [->, shorten >=3pt] (2.45,-.9) to [bend right=35] node[pos=0.3,above] {$f$}(3.7,.25);
\draw [<-, shorten >=3pt] (4,-.75) to [bend left=65] node[pos=0.3,below] {$f$}(3.25,-1);
\draw (.85,-1.15) node {$M^{i}_{0}$};
\draw (1.75,-1.15) node {$\dots M^{i}_{l}$};
\draw (2.35,-1.15) node {$M^{i}_{l+1}$};
\draw (-.45,-1.15) node {$(\bar M,\bar a,\bar N)^i$};
\node at (2.2,.25)[circle, fill, draw, label=315:$ a_{i}$] {};
\node at (1.2,.25)[circle, fill, draw, label=315:$a_{0}$] {};
\begin{scope}
  \clip (0,1) [rrect] rectangle (5,-1);
\draw (-.2,1) [rrect, xslant=-0.4, dotted] rectangle (4, -.75); 
\draw (-.2,1) [rrect, xslant=-0.4, dotted] rectangle (4.7, -.75); 
\end{scope}

\end{tikzpicture}
\caption{The construction of $(\bar M',\bar a,\bar N)$ (dotted) from $(\bar M,\bar a,\bar N)^i$  with $f$ which fixes $\Union_{\beta<\delta}M_\beta$.} \label{fig:reduced tower construction}
\end{figure}
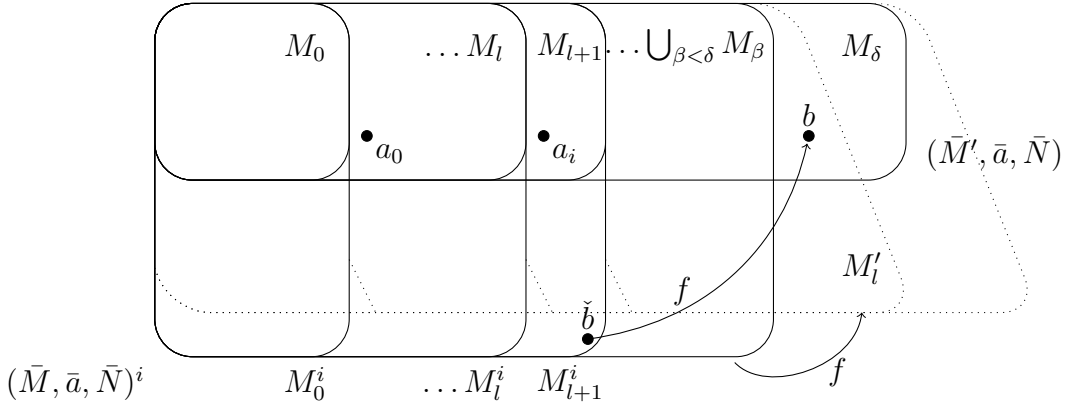
\end{proof}

\section{Symmetry and Reduced Towers}

In this section we introduce a version of symmetry and show that it is equivalent to a statement about reduced towers.

\begin{definition}\label{sym defn}
We say that an abstract elementary class that satisfies Assumption \ref{ss assm} exhibits \emph{symmetry for non-$\mu$-splitting for limit models} if  whenever models $M,M_0,N\in\K_\mu$ and elements $a$ and $b$  satisfy the conditions \ref{limit sym cond}-\ref{last} below, then there exists  $M^b$  a limit model over $M_0$, containing $b$, so that $\tp(a/M^b)$ does not $\mu$-split over $N$.  See Figure \ref{fig:sym}.
\begin{enumerate} 
\item\label{limit sym cond} $M$ is universal over $M_0$ and $M_0$ is a limit model over $N$.
\item\label{a cond}  $a\in M\backslash M_0$.
\item\label{a non-split} $\tp(a/M_0)$ is non-algebraic and does not $\mu$-split over $N$.
\item\label{last} $\tp(b/M)$ is non-algebraic and does not $\mu$-split over $M_0$. 
   
\end{enumerate}

\end{definition}

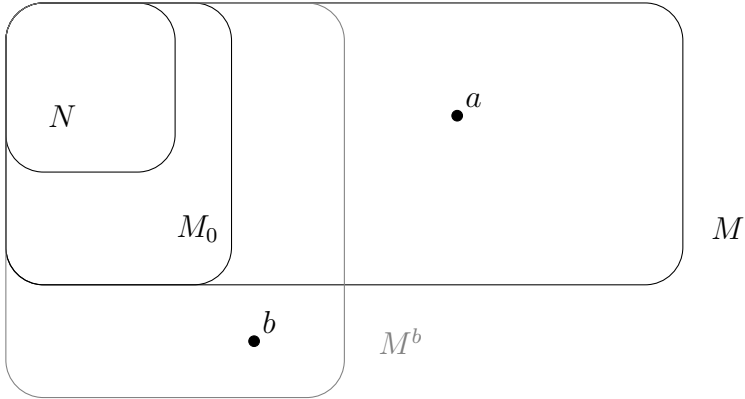
\begin{figure}[h]
\begin{tikzpicture}[rounded corners=5mm, scale=3,inner sep=.5mm]
\draw (0,1.25) rectangle (.75,.5);
\draw (.25,.75) node {$N$};
\draw (0,0) rectangle (3,1.25);
\draw (0,1.25) rectangle (1,0);
\draw (.85,.25) node {$M_0$};
\draw (3.2, .25) node {$M$};
\draw[color=gray] (0,1.25) rectangle (1.5, -.5);
\node at (1.1,-.25)[circle, fill, draw, label=45:$b$] {};
\node at (2,.75)[circle, fill, draw, label=45:$a$] {};
\draw[color=gray] (1.75,-.25) node {$M^{b}$};
\end{tikzpicture}
\caption{A diagram of the models and elements in the definition of symmetry. We assume the type $\tp(b/M)$ does not $\mu$-split over $M_0$ and $\tp(a/M_0)$ does not $\mu$-split over $N$.  Symmetry implies the existence of $M^b$ a limit model over $M_0$ so that $\tp(a/M^b)$  does not $\mu$-split over $N$.} \label{fig:sym}
\end{figure}

The conditions \ref{limit sym cond}-\ref{last} of Definition \ref{sym defn} can be materialized when we make Assumption \ref{ss assm}.  In fact these conditions arise naturally in the standard proofs that reduced towers are continuous.
To achieve conditions \ref{a cond} and \ref{a non-split}, notice that by Assumption \ref{ss assm} and by fixing $M_0$  a limit model,  for every $c\in M\backslash M_0$ there exists $N_c$ so that $M_0$ is a limit model over $N_c$ and $\tp(c/M_0)$ does not $\mu$-split over $N_c$.  So it is relatively easy to find $a$, $M$, $M_0$, and $N$ satisfying the conditions of the definition of $\mu$-symmetry over limit models.  We get examples of $b$ satisfying condition  \ref{last} from Assumption \ref{ss assm} as well by taking $M$ to be a limit model,

This version of symmetry differs slightly from other definitions of symmetry for non-forking in the literature (e.g. \cite{Sh 600} and \cite{BoGr}).  First of all, we only consider symmetry over  limit models.  Our definition of symmetry involves more subtle assumptions, but has a stronger conclusion.  In particular,
we require that $a$ is somewhat ``separated" from $N$ (i.e. that there is a limit model $M_0$ over $N$ inside $M$ which avoids $a$ and $\tp(a/M_0)$ does not $\mu$-split over $N$).  However, in the end we get a stronger conclusion: $\tp(a/M^b)$ does not $\mu$-split over $N$.  This is in contrast to other forms of symmetry (e.g. \cite{Sh 600} and \cite{BoGr}), where the authors could only conclude that $a$ is independent from $M^b$ over $M_0$.   For a detailed comparison of the various notions of symmetry, including this one, see \cite{VV}.

We now provide an equivalent formulation of symmetry for non-$\mu$-splitting over limit models using reduced towers, under the assumption that  $\K$ is $\mu$-superstable.

\begin{theorem}\label{symmetry theorem}
Assume $\K$ is an abstract elementary class satisfying superstability properties for $\mu$ (i.e. Assumption \ref{ss assm}).  Then the following are equivalent:
\begin{enumerate}[a.]
\item\label{sym} $\K$ has symmetry for non-$\mu$-splitting over limit models.
\item\label{red} If $(\bar M,\bar a,\bar N)\in\K^*_{\mu,\alpha}$ is a reduced tower, then $\bar M$ is a continuous sequence (i.e. for every limit ordinal $\beta<\alpha$, we have $M_\beta=\Union_{i<\beta}M_i$).

\end{enumerate}
\end{theorem}

\begin{proof}
We first show (\ref{sym})$\longrightarrow $(\ref{red}).  Suppose $\K$ has symmetry for non-$\mu$-splitting over limit models, but reduced towers are not necessarily continuous.
 Let $(\bar M,\bar a,\bar N)\in\K^*_{\mu,\alpha}$ be a discontinuous reduced tower of minimal length, $\alpha$.  Notice that by Proposition \ref{monotonicity}, we can conclude that $\alpha=\delta+1$ for some limit ordinal $\delta$ and that the failure of continuity must occur at $\delta$.
Let $b\in M_\delta\backslash \Union_{i<\delta}M_i$ witness the discontinuity of the tower.  

By the minimality of $\alpha$ and the density of reduced towers, we can construct a $<$-increasing and continuous chain of reduced, continuous towers $\langle \T^i=(\bar M,\bar a,\bar N)^i\in\K^*_{\mu,\delta}\mid i<\delta\rangle$ with $(\bar M,\bar a,\bar N)^0:=(\bar M,\bar a,\bar N)\restriction \delta$. 
The successor stage of the construction is possible by the density of reduced towers (Fact \ref{density of reduced}) and the limit stages are covered by Fact \ref{union of reduced is reduced}.
By $\delta$-applications of Fact \ref{density of reduced} and Fact \ref{union of reduced is reduced} inbetween successor stages of the construction we can require that for $\beta<\delta$
\begin{equation}\label{limit at successor}
M^{i+1}_{\beta}\text{ is a }(\mu,\delta)\text{-limit over }N_{\beta}.
\end{equation}
Let $\displaystyle{M^\delta_\delta:=\Union_{i<\delta,\;\beta<\delta}M^i_\beta}$.  See Figure \ref{fig:Mdeltas} replacing $\check M$ by $M^\delta_\delta$.  

There are two cases: $1)$ we have $b\in M^\delta_\delta$ and $2)$ we have $b\notin M^\delta_\delta$.  

{\sc Case 1):}  If $b\in M^\delta_\delta$, then we will have found an extension of $(\bar M,\bar a,\bar N)\restriction\delta$  containing $b$ (namely $(\bar M,\bar a,\bar N)^\delta)$) which can easily be lengthened to a discontinuous extension of the entire $(\bar M,\bar a,\bar N)$ tower by taking the $\delta^{th}$ model to be some extension of $M^\delta_\delta$ which is also a limit model universal over  $M_\delta$.  This discontinuous extension of  $(\bar M,\bar a,\bar N)$ along with $b$ witness that $(\bar M,\bar a,\bar N)$ cannot be reduced.

{\sc Case 2):} $b\notin M^\delta_\delta$.  Then $\tp(b/M^\delta_\delta)$ is non-algebraic, and 
by the superstability assumptions there exists $i^*<\alpha$ so that $\tp(b/M^\delta_\delta)$ does not $\mu$-split over $M^{i^*}_{i^*}$.  By monotonicity of non-splitting, we may assume that $i^*$ is a successor and thus by $(\ref{limit at successor})$, $M^{i^*}_{i^*}$ is a $(\mu,\delta)$-limit over $N_{i^*}$.
Now, referring to the Figure \ref{fig:sym}, apply symmetry to $a_{i^*}$ standing in for $a$, $M^{i^*}_{i^*}$ representing  $M_0$,  $N_{i^*}$ as $N$, $M^\delta_\delta$ as $M$, and $b$ as itself.  We can conclude that there exists  $M^b$ containing $b$, a limit model over $M^{i^*}_{i^*}$, for which $tp(a_{i^*}/M^b)$ does not $\mu$-split over $N_{i^*}$.  

Our next step is to consider the tower formed by the diagonal elements in Figure \ref{fig:Mdeltas}.  In particular let $\T^{diag}$ be the 
tower in $\K^*_{\mu,\delta}$ extending $\T\restriction \delta$ whose models are $M^i_i$ for each $i<\delta$.

Define the tower $\T^b\in\K^*_{\mu,i^*+2}$ by the sequences $\bar a\restriction (i^*+1)$, $\bar N\restriction (i^*+1)$ and $\bar M'$ with $M'_j:=M^j_j$ for $j\leq i^*$ and $M'_{i^*+1}:=M^b$.  Notice that $\T^b$ is an extension of $\T^{diag}\restriction(i^*+2)$ containing $b$.  We will explain how we can use this tower to find a tower $\mathring\T^\delta\in\K^*_{\mu,\delta}$ extending $\T^{diag}$ with $b\in \Union_{j<\delta}\mathring M^\delta_{j}$.  This will be enough to contradict our assumption that $\T$ was reduced.

We define  $\langle \mathring\T^j, f_{j,k}\mid i^*+2\leq j\leq k\leq\delta\rangle$ a directed system of towers so that for $j \geq i^*+2$
\begin{enumerate}
\item\label{base} $\mathring\T^{i^*+2}=\T^b$
\item $\mathring\T^j\in\K^*_{\mu,j}$ for $j\leq\delta$
\item $\T^{diag}\restriction j \leq\mathring\T^j$ for $j\leq\delta$
\item $f_{j,k}(\mathring\T^j)\leq\mathring\T^k\restriction j$ for $j\leq k<\delta$
\item\label{id condition} $f_{j,k}\restriction M^{j}_j=id_{M^{j}_j}$ $j\leq k<\delta$
\item\label{limit M'} $\mathring M^{j+1}_{j+1}$ is universal over $f_{j,j+1}(\mathring M^j_j)$ for $j<\delta$
\item\label{b in} $b\in\mathring M^{j}_{j}$ for $j\leq\delta$
\item\label{non splitting} $\tp(f_{j,k}(b)/M^{k}_{k})$ does not $\mu$-split over $M^{i^*}_{i^*}$ for $j<k<\delta$.
\end{enumerate}

We will define this directed system by induction on $k$, with $i^*+2\leq k\leq\alpha$.  The base case $i^*+2$ is determined by condition \ref{base}.  To cover the successor case, suppose that 
$k=j+1$.  
By our choice of $i^*$, we have $\tp(b/\Union_{l<\alpha}M^{l}_l)$ does not $\mu$-split over $M^{i^*}_{i^*}$. 
So in particular by monotonicity of non-splitting, we notice:
\begin{equation}\label{Mjj non-split}
\tp(b/M^{j+1}_{j+1})\text{ does not }\mu\text{-split over }M^{i^*}_{i^*}.
\end{equation} 
 Using the definition of towers, the choice of $i^*$, and the fact that $M^{j+1}_{j+1}$ was chosen to be a $(\mu,\delta)$-limit over $N_{j+1}$, we can apply symmetry to $a_{j+1}$, $M^{j+1}_{j+1}$, $ \Union_{l<\delta}M^{l}_l$, $b$ and $N_{j+1}$ which will yield $M^b_{j+1}$ a  limit model over $M^{j+1}_{j+1}$ containing $b$ 
 so that $\tp(a_{j+1}/M^b_{j+1})$ does not $\mu$-split over $N_{j+1}$ (see Figure \ref{fig:successor}).

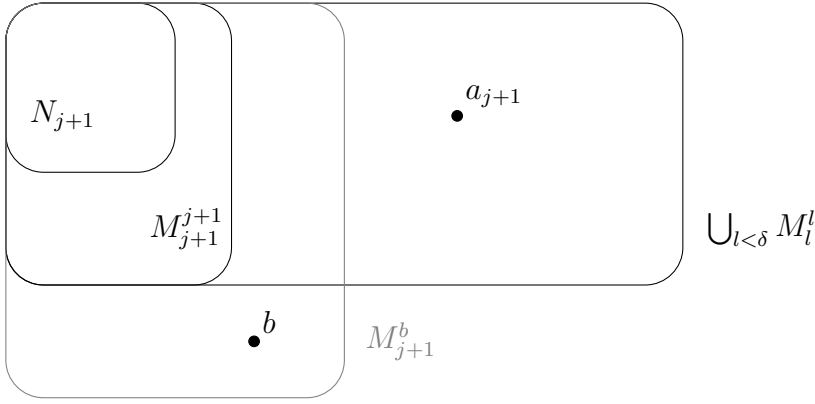
\begin{figure}[h]
\begin{tikzpicture}[rounded corners=5mm, scale=3,inner sep=.5mm]
\draw (0,1.25) rectangle (.75,.5);
\draw (.25,.75) node {$N_{j+1}$};
\draw (0,0) rectangle (3,1.25);
\draw (0,1.25) rectangle (1,0);
\draw (.8,.25) node {$M^{j+1}_{j+1}$};
\draw (3.35, .25) node {$\Union_{l<\delta}M^{l}_l$};
\draw[color=gray] (0,1.25) rectangle (1.5, -.5);
\node at (1.1,-.25)[circle, fill, draw, label=45:$b$] {};
\node at (2,.75)[circle, fill, draw, label=45:$a_{j+1}$] {};
\draw[color=gray] (1.75,-.25) node {$M^{b}_{j+1}$};
\end{tikzpicture}
\caption{A diagram of the application of symmetry in the successor stage of the directed system construction in the proof of (\ref{sym})$\longrightarrow$(\ref{red}) in Theorem \ref{symmetry theorem}. We have $\tp(b/ \Union_{l<\delta}M^{l}_l)$ does not $\mu$-split over $M^{j+1}_{j+1}$ and $\tp(a_{j+1}/M^{j+1}_{j+1})$ does not $\mu$-split over $N_{j+1}$.  Symmetry implies the existence of  $M^b$ a limit model over $M^{j+1}_{j+1}$. so that $\tp(a_{j+1}/M^b)$  does not $\mu$-split over $N_{j+1}$.} \label{fig:successor}
\end{figure}

Fix $M'$ to be a model of cardinality $\mu$ extending  both $\mathring M^j_j$ and $M^{j+1}_{j+1}$
Since $M^b_{j+1}$ is a limit model over $M^{j+1}_{j+1}$, there exits $f_{j,j+1}:M'\rightarrow M^b_{j+1}$ with $f_{j,j+1}=id_{M^{j+1}_{j+1}}$ so that $M^b_{j+1}$ is also universal over $f_{j,j+1}(\mathring M^j_j)$.  Notice that condition \ref{non splitting} of the construction is satisfied because of (\ref{Mjj non-split}), invariance, and our choice of $f_{j,j+1}\restriction M_{j+1}^{j+1}=\id$.
Therefore, it is easy to check that $\mathring \T^{j+1}$ defined by the models $\mathring M^{j+1}_l:=f_{j,j+1}(\mathring M^j_l)$ for $l\leq j$ and $\mathring M^{j+1}_{j+1}:=M^b_{j+1}$ are as required.
Then the rest of the directed system can be defined by the induction hypothesis and the mappings $f_{l,j+1}:=f_{l,j}\circ f_{j,j+1}$ for $i^*+2\leq l<j$.

Now consider the limit stage  $k$ of the construction.  First,  let $\grave \T^k$ and $\langle\grave f_{j,k}\mid i^*+2\leq j<k\rangle$ be a direct limit of the system defined so far.  We use the $\grave{}$ notation since these are only approximations to the tower and mappings that we are looking for.  We will have to take some care to find a direct limit that contains $b$ in order to satisfy Condition \ref{b in} of the construction.
By Fact \ref{limit stage prop} and our induction hypothesis, we may choose this direct limit so that for all $j<k$
\begin{equation*}
\grave f_{j,k}\restriction M^{j}_j=id_{M^{j}_j}.
\end{equation*}
Consequently $\grave M^\alpha_j:=\grave f_{j,k}(\mathring M^j_j)$ is universal over $M^{j}_j$, and $\Union_{j<k}\mathring M^k_j$ is a limit model witnessed by condition \ref{limit M'} of the construction.  Additionally, because $\T^{diag}\restriction k$  is continuous,
 the tower  $\grave\T^k$   composed of the models $\grave M^k_j$, extends $\T^{diag}\restriction k$.

We will next show that for every $j<k$,
\begin{equation}\label{limit non split eqn}
\tp(\grave f_{i^*+2,k}(b)/M^j_j)\text{ does not }\mu\text{-split over }M^{i^*}_{i^*}.
\end{equation}
To see this, recall that for every $j<k$, by the definition of a direct limit, $\grave f_{i^*+2,k}(b)=\grave f_{j,k}(f_{i^*+2,j}(b))$.
By condition \ref{non splitting} of the construction, we know
\begin{equation*}
\tp(f_{i^*+2,j}(b)/M^{j}_{j})\text{ does not }\mu\text{-split over }M^{i^*}_{i^*}.
\end{equation*}
Applying $\grave f_{j,k}$ to this implies $\tp(\grave f_{i^*+2,k}(b)/M^j_j)$ does not $\mu$-split over $M^{i^*}_{i^*}$, establishing $(\ref{limit non split eqn})$.

Because $M^{j+1}_{j+1}$ is universal over $M^j_j$ by construction, we can apply
Assumption \ref{ss assm} to $(\ref{limit non split eqn})$ yielding
\begin{equation}\label{grave f}
\tp(\grave f_{i^*+2,k}(b)/\Union_{j<k}M^j_j)\text{ does not }\mu\text{-split over }M^{i^*}_{i^*}.
\end{equation}

Because $\grave f_{i^*+2,k}$ fixes $M^{i^*+1}_{i^*+1}$,  $\tp(b/M^{i^*+1}_{i^*+1})=\tp(\grave f_{i^*+2,k}(b)/M^{i^*+1}_{i^*+1})$.
We can then apply the uniqueness of non-splitting extensions to $(\ref{grave f})$ to see that  $\tp(\grave f_{i^*+2,k}(b)/\Union_{j<k}M^j_j)=\tp(b/\Union_{j<k}M^j_j)$.  Thus we can fix $g$ an automorphism of the monster model fixing $\Union_{j<k}M^j_j$ so that $g(\grave f_{i^*+2,k}(b))=b.$

We will then define $\mathring \T^k$ to be the tower $g(\grave\T^k)$ and the mappings for our directed system will be $f_{j,k}:=g\circ\grave f_{j,k}$ for all $ i^*+2\leq j<k$.
This completes the construction.

Now that we have $\mathring\T^\delta$ a tower extending $\T\restriction\delta$ which contains $b$, we are in a situation 
similar to the proof in case $1)$.  
To contradict that $\T$ is reduced, 
we need only 
lengthen $\mathring\T^\delta$ to a discontinuous extension of the entire $(\bar M,\bar a,\bar N)$ tower by taking the $\delta^{th}$ model to be some extension of $\Union_{i<\delta}\mathring M^i_i$ which is also universal over  $M_\delta$.  This discontinuous extension of  $(\bar M,\bar a,\bar N)$ along with $b$ witness that $(\bar M,\bar a,\bar N)$ cannot be reduced.

%
%
%

Next we show (\ref{red})$\longrightarrow$ (\ref{sym}).  
As in the definition of $\mu$-symmetry over limit models, 
suppose $M$ is universal over $M_0$ and $M_0$ is a limit model over $N$.
 Fix $b$ so that the non-algebraic $\tp(b/M)$ does not $\mu$-split over $N$.  Fix $a\in M\backslash M_0$. Without loss of generality, by monotonicity of non-splitting, we may assume that $M$ is a limit model over $M_0$.  Let $\langle M_i\mid i<\delta\rangle$ witness this. We can arrange that $M_{i+1}$ is a limit model over $M_i$ and $a\in M_1$.  We will find $M^b$ a limit model over $M_0$ containing $b$ and extending $N$ so that $\tp(a/M^b)$ does not $\mu$-split over $N$.  

We start out by  building a tower of length $\delta+1$.  Use the models in the sequence $\langle M_i\mid i<\delta\rangle$ as the first part of the tower, and define $M_{\delta}$ to be some limit model extending $M$ containing $b$.  
Set $a_0:=a$, and for $0<i<\delta$ choose $a_i\in M_{i+1}\backslash M_i$ realizing the extension of $\tp(a/M_0)$ to $M_i$ that does not $\mu$-split over $N$.
Then set  $N_i:=N$ for each $i$.  Refer to the tower of length $\delta+1$ defined this way as $\mathcal T$. 

Notice that $\mathcal T$ is discontinuous at $\delta$; therefore by our assumption, it is not reduced.  
By Fact \ref{density of reduced} we can find $\mathcal T'$ extending $\mathcal T$ so that $\mathcal T'$ is reduced.  By our hypothesis (\ref{red}), we know that $\mathcal T'$ must be continuous.
  Since $b$ appears in this continuous tower,  there exists $j<\delta$ so that $b\in M'_{j}$.  Fix the minimal such $j$ and denote it by $j^*$.  There are two cases to consider.

{\sc Case 1:}  $j^*=0$.  By definition of the ordering on towers, since $\mathcal T<\mathcal T'$, we know that $\tp(a_{0}/M'_{0})$ does not $\mu$-split over $N$.  Thus $M'_{0}$ is as required.

{\sc Case 2:} $j^*>0$.  
By the choice of $a_j$ and uniqueness of non-splitting extensions, we know $\tp(a_{0}/M'_{0})=\tp(a_{j^*}/M'_{0})$.  Thus, there exists $f\in\Aut_{M_0}(\C)$ with $f(a_{j^*})=a_{0}$.  Since $M_1$ is universal over $M_0$, we can also require that our choice of $f$ has the property that $f\restriction M:M \rightarrow_{M_0} M_{1}$.
Because $\tp(b/M)$ does not $\mu$-split over $N$, we know 
\begin{equation*}\label{b non-split}
\tp(f(b)/f(M))=\tp(b/f(M)).
\end{equation*}
This implies there exists an automorphism $g$ of the monster model fixing $f(M)$ so that $g(f(b))=b$.  

We claim that $M^b:=g(f(M'_{j^*}))$ is as required.  First notice that $b\in M^b$ since $f(b)\in f(M'_{j^*})$ and $g(f(b))=b$.  Next we need to check that $\tp(a_{0}/M^b)$ does not $\mu$-split over $N$.  By the definition of towers, 
\begin{equation*}\label{non-split j^*} \tp(a_{j^*}/M'_{j^*})\text{ does not }\mu\text{-split over }N_{j^*}(=N).  \end{equation*}
By invariance and by our choice of $f$ and $g$ fixing $N$ with $g(f(M'_{j^*}))=M^b$, we can conclude that  
\begin{equation*}
\tp(g(f(a_{j^*}))/M^b)\text{ does not }\mu\text{-split over }N.
\end{equation*} 
By our choice of $f$ taking $a_{j^*}$ to $a_{0}$, we get 
\begin{equation}\label{non-split gf} \tp(g(a_{0})/M^b)\text{ does not }\mu\text{-split over }N.  \end{equation}
Because $g$ fixes $f(M)$ and $a_{0}=f(a_{j^*})\in f(M)$, $(\ref{non-split gf})$ implies that $\tp(a_{0}/M^b)$ does not $\mu$-split over $N$ as required.

\end{proof}

\section{Concluding Remarks}

Corollary \ref{categoricity corollary} appeared in a previous version of this paper, and we thank Sebastien Vasey for pointing out a clearer proof that is included below.

\begin{proof}[Proof of Corollary \ref{categoricity corollary}]

Notice that the results in \cite{Sh 394} demonstrate that categoricity in $\mu^+$ implies that the assumptions of this paper are met.   Namely $\K$ is stable in $\mu$ and non-splitting satisfies the continuity and locality properties for $\mu$.  
By stability in $\mu$, the model of cardinality $\mu^+$ is saturated.  Therefore by categoricity in $\mu^+$, the hypotheses of Theorem \ref{transfer theorem} are met, and by Theorem \ref{symmetry theorem}, we get $\mu$-symmetry over limit models.

\end{proof}
In \cite{Va3}, where symmetry is further explored, we weaken the assumption in Corollary \ref{categoricity corollary} that the categoricity cardinal $\lambda=\mu^+$.

%
%

Building on much of Shelah's work including \cite{Sh book}, Vasey derives a non-forking relation without assuming tameness for  $\mu^+$-categorical abstract elementary classes that satisfy the amalgamation property and have no maximal models  (Theorem 1.2 and section 3 of \cite{V2}).  Vasey points out that Corollary \ref{categoricity corollary} contributes to his work since it implies that the non-forking relation additionally satisfies symmetry over limit models.  
In the terminology of his paper, we state this result.

\begin{corollary}
Suppose that $\K$ is an abstract elementary class that is categorical in a sufficiently large $\mu^+$, satisfies the amalgamation and joint embedding properties, and has no maximal models. 
Then the relation ``explicitly non-forking" satisfies symmetry over limit models. 
\end{corollary}
In later sections of his paper \cite{V2}, Vasey derives a  symmetric non-forking relation without requiring that the categoricity cardinal is a successor \cite{V2}, but at the expense of   assuming tameness.  

As mentioned in the introduction, the results of this paper have been used to make further progress on the understanding of categorical, $\mu$-superstable, and even strictly $\mu$-stable AECs.  We refer the reader to a few applications and extensions of the results  presented here: \cite{VV}, \cite{Va3}, \cite{BV}, and \cite{Vas-categoricity}.

On a final note, with a little care, the arguments presented in this paper can also be carried out in the context of \cite{Va2} where amalgamation is not assumed \cite{V-Volume}.





\bibliographystyle{model1-num-names}
\bibliography{<your-bib-database>}



\end{document}